\documentclass{amsart}
\usepackage{amsmath,amsthm,amsfonts,amssymb,amscd}
\usepackage[all]{xy}

\usepackage[english]{babel}
\usepackage{graphicx,color}
\usepackage[T1]{fontenc}
\usepackage[latin1]{inputenc}
\usepackage{ae}

\usepackage[mathscr]{eucal}
\usepackage{amsmath,amssymb}
\usepackage{graphics}
\usepackage{multirow}
\usepackage{hyperref}
\usepackage[active]{srcltx}

\newtheorem{thm}{Theorem}[section]
\newtheorem{THM}{Theorem}

\newtheorem{prop}[thm]{Proposition}

\newtheorem{lemma}[thm]{Lemma}

\theoremstyle{definition}

\newtheorem{remark}[thm]{Remark}

\newcommand{\nc}{\newcommand}
\nc {\hh}{\check{h}}
\nc {\DD}{\mathcal{D}}
\nc {\RR}{\mathcal{R}}
\nc {\Pp}{\mathbb{P}}
\nc {\Ss}{\mathcal{S}}
\nc {\PP}{\mathbb{P}^{2}}
\nc {\Pd}{ \check{\mathbb{P}}^{2}}
\nc {\WW}{\mathcal{W}}
\nc {\Sym}{\mathrm{Sym}}
\nc {\OO}{\mathcal{O}}
\nc {\CC}{\mathbb{C}}
\nc {\EE}{\mathcal{E}}
\nc {\MM}{\mathcal{M}}
\nc {\KK}{\mathcal{K}}
\nc {\PW}{\mathcal{P}}
\nc {\NW}{\mathcal{N}_{\WW}}
\nc {\FF}{\mathcal{F}}
\nc {\GG}{\mathcal{G}}
\nc {\ZZ}{\mathcal{Z}}
\nc {\LL}{\mathcal{L}}
\nc {\HH}{\mathcal{H}}
\nc {\NN}{\mathcal{N}}
\nc {\VV}{\mathcal{V}}
\nc {\Ww}{\mathbb{W}}
\nc {\QQ}{\mathbb{Q}}
\nc {\II}{\mathcal{I}}
\nc {\tang}{\mathrm{Tang}}
\nc {\Diff}{\mathrm{Diff}}
\nc {\Diffun}{\mathrm{Diff}^1(\CC^2,0)}
\nc {\pr}{\mathrm{pr}}
\nc {\cod}{\mathrm{codim}}
\nc {\id}{\mathrm{id}}
\nc {\ord}{\mathrm{ord}}
\nc {\Fol}{\mathrm{Fol}}
\nc {\rank}{\mathrm{rank}}
\nc {\Lie}{\mathrm{Lie}}
\nc {\Z}{\mathbb{Z}}

\begin{document}

\title{Neighborhoods of rational curves without functions}
\author[M. Falla Luza, F. Loray]
{Maycol Falla Luza$^{1}$, Frank Loray$^2$}
\address{\newline $1$ Universidade Federal Fluminense, Rua M\'ario Santos Braga S/N, Niter\'oi, RJ, Brazil.\hfill\break
$2$ Univ Rennes, CNRS, IRMAR - UMR 6625, F-35000 Rennes, France.} 
\email{$^1$ maycolfl@gmail.com} \email{$^2$ frank.loray@univ-rennes1.fr}
\thanks{
The second author is supported by CNRS, Henri Lebesgue Center and ANR-16-CE40-0008 project ``{\it Foliage}''. 
The authors also thank Brazilian-French Network in Mathematics and CAPES-COFECUB Project Ma 932/19
``{\it Feuilletages holomorphes et int\'eractions avec la g\'eom\'etrie}''.
We finally thank Jorge Vit\'orio Pereira for helpfull discussions on the subject.}
\date{\today}

\subjclass{} \keywords{Foliation, Projective Structure, Rational Curves}

\begin{abstract}
We prove the existence of (non compact) complex surfaces with a smooth rational curve embedded such that there does not 
exist any formal singular foliation along the curve. In particular, at arbitray small neighborhood of the curve,  
any meromorphic function is constant. This implies that the Picard group is not countably generated.
\end{abstract}

\maketitle

\section{Introduction}

Let $S$ be a complex surface and $\iota:\mathbb P^1\hookrightarrow\subset S$ be an embedded smooth rational curve.
Denote by $C\cdot C\in\mathbb Z$ the self-intersection of the image curve $C=\iota(\mathbb P^1)$: 
it coincides with the degree $d$ of the normal bundle $N_C:=T_S/T_C$, i.e. $\iota^*N_C=\mathcal O_{\mathbb P^1}(d)$ 
where $d:=C\cdot C$.
We are interested in properties of the germ of neighborhood $(S,C)$, i.e. of any sufficiently small neighborhood of $C$
in the surface $S$. For instance, we say that $(S,C)$ is linearizable if any sufficiently small neighborhood $C\subset V\subset S$
is equivalent to a neighborhood of the zero section in the total space of $N_C$.

\begin{figure}[ht!]
  \centering
    \includegraphics[scale=0.4]{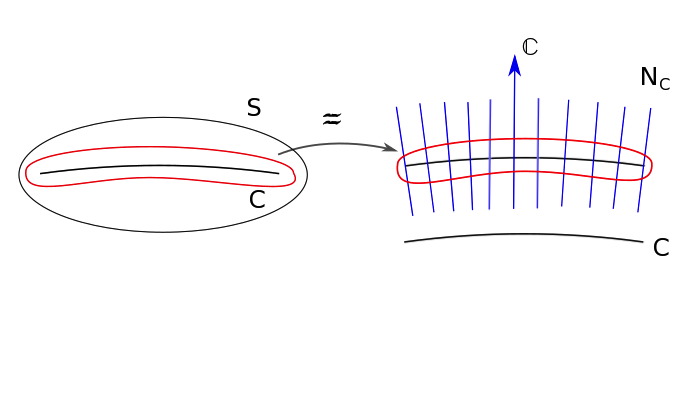}
  \caption{Linearization}
  \label{fig:linealizacion}
\end{figure}

It has been well-known, since the celebrated work of Grauert  \cite{Grauert}, that the neighborhood germ $(S,C)$ is linearizable
whenever the self-intersection number of the curve is negative: $C\cdot C<0$. In that case, the curve can be contracted to a point:
there is a morphism $\phi:(S,C)\to(S',p)$ sending $C$ to the point $p$ of a germ of analytic surface $S'$ (which is singular at $p$
if $d<1$). In particular, the ring of germs of holomorphic functions at $C$ is not
finitely generated, comparable with the ring $\mathbb C\{X,Y\}$ of convergent power-series. 

When $C\cdot C=0$, it follows from Kodaira's Deformation Theory \cite{Kodaira} (and Fischer-Grauert Theorem \cite{FischerGrauert}) 
that the neighborhood
is again linearizable, i.e. a product $C\times(\mathbb C,0)$ (see \cite{Savelev}). 
The ring of holomorphic function germs $\mathcal O(S,C)\simeq\mathcal O(\mathbb C,0)$ consists of function germs 
the transversal factor $(\mathbb C,0)$, i.e. isomorphic to the ring $\mathbb C\{X\}$.

On the other hand, when $C\cdot C>0$ is fixed, we have no more rigidity: the deformation space is infinite dimensional.
More precisely, the analytic classification of such neighborhoods up to biholomorphisms is parametrized by some 
space of diffeomorphisms comparable with $\mathbb C\{X,Y\}$ as shown by Mishustin \cite{Mishustin}
(see also \cite{FrankMaycol}). Such a neighborhood is pseudo-concave, and any holomorphic function
is constant; following Andreotti \cite{Andreotti} the field of meromorphic function germs $\mathcal M(S,C)$ 
has transcendance degree $\le 2$, i.e. is at most a finite extension of the field $\mathbb C(X,Y)$ of rational functions.

The goal of this note is to prove the existence of neighborhoods with $C\cdot C>0$ without non constant meromorphic functions.

\begin{THM}\label{Neighborhood without function}
For each $d\in\mathbb Z_{>0}$, there exists a germ of surface neighborhood $(S,C)$ with $C$ rational, $C\cdot C=d$,
and such that any meromorphic function germ $f\in\mathcal M(S,C)$ is constant.
\end{THM}

In other words, the algebraic dimension of $(S,C)$ is zero. This result is quite surprising since it follows 
from Kodaira \cite{Kodaira} that the neighborhood is covered by rational curves. Precisely, the curve $C$ itself 
can be deformed in a $(d+1)$-dimensional family of rational curves $(\mathbb C^{d+1},0)\ni t\mapsto C_t\subset S$. 
For any two rational curves $C_{t_1},C_{t_2}$,
the corresponding line bundles $\mathcal O_S([C_{t_i}])$ must be non isomorphic, since otherwise the isomorphism 
would provide a meromorphic function with divisor $[C_{t_1}]-[C_{t_2}]$, therefore non constant.
In particular, the Picard group $\mathrm{Pic}(S,C)$ of germs of line bundles cannot be generated by a countable basis.

A non constant meromorphic function $f:(S,C)\dashrightarrow\mathbb P^1$ induces a singular holomorphic
foliation on the neighborhood germ $(S,C)$, which is singular at indeterminacy points of $f$, and whose leaves
are (connected components of) the fibers of $f$ outside of the singular set.
Then Theorem \ref{Neighborhood without function} is a direct consequence of the following.

\begin{THM}\label{Neighborhood without foliation}
For each $d\in\mathbb Z_{>0}$, there exists a germ of surface neighborhood $(S,C)$ with $C\cdot C=d$
and without any singular foliation.
\end{THM}

To prove Theorem \ref{Neighborhood without foliation}, we first reduce to a singular but simpler case as follows.
We note that after blowing-up $d+1$ distinct points on $C$, we get a neighborhood of $C'\cup E_1\cup\cdots\cup E_{d+1}$
where $C'$ is the strict transform of $C$ and $E_i$'s are the exceptional divisors. Moreover, all these $d+2$ rational curves
have self-intersection $-1$. After contracting $C'$, we get a neighborhood of $d+1$ rational curves $D_1\cup\cdots\cup D_{d+1}$ intersecting at $1$ point and having self-intersection $0$. If the neighborhood germ of $D_1\cup D_2$ inside this larger neighborhood 
germ does not carry a singular foliation, then we are done. This latter problem is simpler as $0$-neighborhoods are trivial:
a foliation in $\mathbb P^1\times(\mathbb C,0)$ is globally defined by a Pfaffian equation $\omega=0$ where the $1$-form $\omega$
takes the form
$$\omega=P(x,y)dx+Q(x,y)dy$$
where $P,Q$ are polynomials in $x$-variable. The idea of the proof is that, once we fix the degree of these polynomial, 
we observe that the number of entries in the patching map between the two neighborhoods grows fast in the jet spaces
that the entries of the foliations on both neighborhoods.
We actually prove that generic neighborhoods have no foliations.

\begin{figure}[ht!]
  \centering
    \includegraphics[scale=0.4]{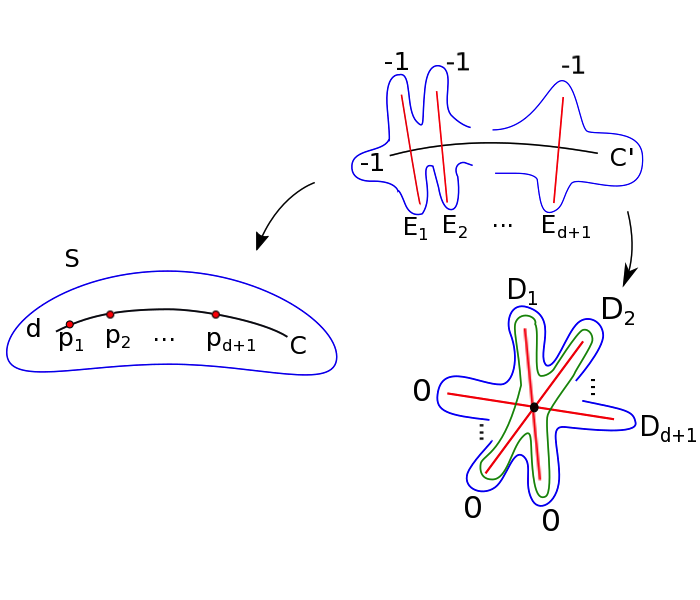}
  \caption{Reduction to 0-Neighborhoods}
  \label{fig:linealizacion}
\end{figure}

Following \cite{Hitchin1}, there is a duality between $1$-neighborhoods rational curves up to (semi-local) biholomorphisms,
and local Cartan projective structures on $(\mathbb C^2,0)$ up to local biholomorphisms. These latter ones consist in 
a collection of unparametrized geodesics; they can be defined as solutions $y(x)$ of a second-order differential equation
$$y''=a(x,y)+b(x,y)y'+c(x,y)(y')^2+d(x,y)(y')^3$$
with $a,b,c,d$ holomorphic once a system of coordinates $(x,y)\in(\mathbb C^2,0)$ has been choosen. 
A particular case, investigated in \cite{Hurtubise} from this duality point of view, is the case of Painlev\'e equations,
for instance the first one $y''=6 y^2+x$. We expect that dual $1$-neighborhoods of Painlev\'e equations have 
generic behaviour, in particular carrying no non-constant meromorphic functions.

\section{Glueing two 0-neighborhoods}

Following the blowing-up/contraction procedure explained in the introduction (see Figure \ref{fig:linealizacion}), 
in order to prove Theorem \ref{Neighborhood without foliation}, 
it is enough to prove the non existence of singular foliations for neighborhood germ $(V,D)$
of some bouquet $D=D_1\cup\cdots\cup D_{d+1}$ of rational curves having self-intersection $0$ in $V$, 
and intersecting transversally at a single point.
The resulting $(+d)$-neighborhood $(S,C)$, obtained by the above procedure, will admit no singular foliation as well.

As a germ, such a neighborhood $(V,D)$ can be obtained by gluing $0$-neighborhoods $(V_i,D_i)$, that are trivial by \cite{Savelev},
i.e. equivalent to $\mathbb P^1\times(\mathbb C, 0)$. A single $0$-neighborhood admits many foliations and meromorphic
functions. We prove that this is no more true for a bouquet of several $0$-neighborhoods. It is clearly enough
to prove this for a bouquet of two $0$-neighborhoods. Indeed, if we can glue two $0$-neighborhoods such that 
the resulting germ admits no foliation, then the same will hold true for any neighborhood constructed by patching
additional $0$-neighborhoods.

\subsection{Some definitions}
A neighborhood germ $(V,D)$ of $D=D_1\cup D_2$ as above can be described as follows.
Consider $S_0=\mathbb P^1\times\mathbb P^1$ viewed as compactification of $\mathbb C^2$
with coordinate $(x,y)$. Define the curve $D\subset S_0$ as the closure
of $xy=0$. Let $D_1$ (resp. $D_2$) be the component defined by $y=0$ (resp. $x=0$). 
Denote by $(V_i,D_i)$ a copy of the neighborhood germ $(S_0,D_i)$ for $i=1,2$, 
and by $p_i\in D_i$ the point corresponding to $(x,y)=0$.
Finally, given a biholomorphism germ $\Phi: (\mathbb{C}^2,0)\rightarrow (\mathbb{C}^2,0)$,
we define by $(V_\Phi,D)$ the neighborhood germ obtained by gluing 
$$(V_2,D_2)\stackrel{\Phi}{\longrightarrow}(V_1,D_1).$$
For instance, $(V_{\text{id}},D)\simeq(S_0,D)$.
A general neighborhood germ $(V,D)$, with each $D_i\cdot D_i=0$ for $i=1,2$,
is isomorphic to $(V_\Phi,D)$ for some $\Phi\in\Diff(\CC^2,0)$.
Moreover, playing automorphims of each $(V_i,D_i)$, we can assume $\Phi$ tangent to the identity,
i.e.
$$\Phi\in\Diffun=\left\{(\phi,\psi)=\left(x+\sum_{i+j>1}a_{ij}x^iy^j,\ y+\sum_{i+j>1}b_{ij}x^iy^j\right)\ \in\mathbb C\{x,y\}^2\right\}.$$

\begin{figure}[ht!]
  \centering
    \includegraphics[scale=0.4]{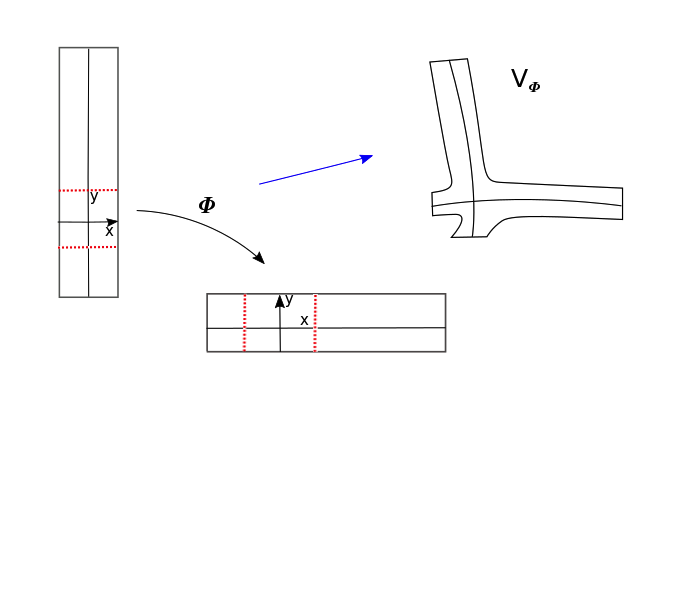}
  \caption{Gluing two 0-Neighborhoods}
  \label{fig:TwoLines}
\end{figure}

\subsection{Singular foliations}
A foliation $\mathcal F$ on the neighborhood germ $(V_{\Phi},D)$ is defined by its tangent sheaf $T_{\mathcal F}\subset TV_{\Phi}$,
which is itself defined in each chart $V_i$ by $\ker(\omega_i)$ where $\omega_i$ is a germ of non trivial meromorphic $1$-form.
The compatibility condition between the two charts writes 
$$\Phi^*\omega_1 \wedge \omega_2 =0.$$
For instance, if $f$ is a non constant meromorphic function on $V_{\Phi}$, then $df$ is the $1$-form defining such a foliation.

\begin{lemma}\label{Lem:polynomial1forms}
Maybe multiplying each $\omega_i$ by a meromorphic function, we may assume
$$\omega_i=P_i(x,y)dx+Q_i(x,y)dy$$
with 
\begin{itemize}
\item $P_1,Q_1\in\CC\{y\}[x]$ (polynomial in $x$),
\item $P_2,Q_2\in\CC\{x\}[y]$ (polynomial in $y$).
\end{itemize}
Moreover, we can take $P_i$, $Q_i$ with no common factor.
\end{lemma}

\begin{proof}
If the restriction $\mathcal F\vert_{V_2}$ coincides with the vertical fibration $x=\text{constant}$, then we can set $\omega_2=dx$.
If not, then $\mathcal F\vert_{V_2}$ is globally defined by $\frac{dy}{dx}=f_2(x,y)$
for a global meromorphic function $f_2$ on $V_2$. This latter one must be rational in $y$-variable, i.e. can be decomposed
as $f_2=-P_2/Q_2$ where $P_2,Q_2\in\CC\{x\}[y]$; therefore, $\mathcal F\vert_{V_2}$ is defined by $\ker(\omega_2)$
with $\omega_2=P_2 dx+Q_2 dy$. A similar argument gives the decomposition for $\omega_1$. 
\end{proof}

\begin{remark}A formal foliation on $(V_\Phi,D)$ is defined in charts by $\ker(\omega_i)$ with
$\omega_i=P_i(x,y)dx+Q_i(x,y)dy$ and 
\begin{itemize}
\item $P_1,Q_1\in\CC[[y]][x]$ (polynomial in $x$, formal in $y$),
\item $P_2,Q_2\in\CC[[x]][y]$ (polynomial in $y$, formal in $x$).
\end{itemize}
In fact, we will not use the convergence of $\omega_i$ along our construction and prove
the non existence of formal foliations. However, thinking back to $(+d)$-neighborhoods
of smooth rational curves $(S,C)$, it is known that formal foliations are actually analytic
(see \cite[Property G3]{HM}).
\end{remark}

For a foliation $\mathcal F$ defined in charts by $\ker(\omega_i)$ as in Lemma \ref{Lem:polynomial1forms}, we define
\begin{itemize}
\item $k(\mathcal F)=\mathrm{Max}\left\{\deg_x\omega_1,\deg_y\omega_2\right\}$ the ``degree'' of the foliation,
\item $\nu(\mathcal F)=\mathrm{Min}\left\{\ord_{(x,y)}(P_1),\ord_{(x,y)}(Q_1)\right\}=\mathrm{Min}\left\{\ord_{(x,y)}(P_2),\ord_{(x,y)}(Q_2)\right\}$ the multiplicity (or vanishing order) at $0$.
\end{itemize}
We say that $\mathcal F$ is of type $(k,\nu)$.

\section{Obstructions for $(k,\nu)$-foliations}

The goal of this section is to prove that a generic perturbation $(V_\Phi,D)$ of a given neighborhood $(V_{\Phi_0},D)$
does not admit foliations of type $(k,\nu)$. In fact, a generic perturbation of a finite set of coefficients of $\Phi_0$,
of arbitrary large order, will provide such a $\Phi$. In order to settle the precise statement, let us introduce some 
definitions of jets of diffeomorphisms.

\subsection{Some notations}
For $N>0$ natural number, we write 
$$J^N\sum_{i,j\ge0}a_{ij}x^iy^j=\sum_{i,j\ge0}^{i+j\le N}a_{ij}x^iy^j$$ 
the jet of order $N$ of a power-series. For $N_0<N_1$, we will denote 
$$J_{N_0}^{N_1}\sum_{i,j\ge0}a_{ij}x^iy^j=\sum_{N_0<i+j\le N_1}a_{ij}x^iy^j$$ 
the truncation between degree $N_0$ and degree $N_1$. For $\Phi=(\phi,\psi)\in\Diffun$, we similarly define
$$J^N\Phi=(J^N\phi,J^N\psi)\ \ \ \text{and}\ \ \ 
J_{N_0}^{N_1}\Phi=(J_{N_0}^{N_1}\phi,J_{N_0}^{N_1}\psi)$$
and we define by $J^N\Diffun$ and $J_{N_0}^{N_1}\Diffun$ the corresponding sets of truncated diffeomorphisms.
Finally, if we define 
$$\OO=\CC\{x,y\}\ \ \ \text{and}\ \ \ \Omega^1=\OO\cdot dx+\OO\cdot dy$$
and 
$$\Omega_1^{(k,\nu)}=\left\{\omega_1\in\Omega^1\ ;\ \deg_x(\omega_1)\le k,\ \ord_{(x,y)}(\omega_1)=\nu\right\}.$$
We define in a similar way
$$\Omega_2^{(k,\nu)}=\left\{\omega_2\in\Omega^1\ ;\ \deg_y(\omega_2)\le k,\ \ord_{(x,y)}(\omega_2)=\nu\right\}$$
as well as $J^N\OO$ and  $J^N\Omega_i^{(k,\nu)}$ in the obvious way.

\subsection{Obstruction Lemmae}

This section is devoted to prove the:

\begin{prop}\label{important proposition}
For any $k,\nu,N_0\in\Z_{\ge0}$, and $\Phi_0 \in J^{N_0}\Diffun$,
there exist a non empty Zariski open set $U \subset J^{N_1}_{N_0}\Diffun$ for some $N_1>N_0$, 
such that:
\begin{itemize}
\item[] if $\Phi \in \Diffun$ satisfies $J^{N_0}\Phi =\Phi_0$ and $J^{N_1}_{N_0}\Phi \in U$, \\
 then $V_{\Phi}$ does not have any formal foliation of type $(k, \nu)$. 
\end{itemize}
\end{prop}

\begin{remark}As we shall see, $N_1$ does not depend on $\Phi_0$, but $U$ does. 
\end{remark}

Given a gluing diffeomorphism $\Phi$ and foliations of type $(k,\nu)$ defined by $\ker(\omega_i)$
on $(V_i,D_i)$, for $i=1,2$, then the compatibility condition 
$$\Phi^*\omega_1 \wedge \omega_2=f(x,y)\cdot dx\wedge dy$$
gives a germ of power-series $f\in\OO$ that is identically zero if, and only if, 
the two foliations patch into a global foliation on $(V_\Phi,D)$. Consider the map
$$\begin{matrix}\Diffun\times\Omega_1^{(k,\nu)}\times\Omega_2^{(k,\nu)}&\stackrel{E}{\longrightarrow}&\OO;\\ 
(\Phi,\omega_1,\omega_2)&\mapsto &f.\end{matrix}$$
If we denote by $Z=E^{-1}(0)$ the preimage of the zero power-series, at the end, we will 
prove that the projection $W=\pr_1(Z)$ of $Z$ towards $\Diffun$ is not onto. In other word, there exist
$\Phi\in\Diffun\setminus W$ which, by definition, are such that $V_\Phi$ carries no $(k, \nu)$-foliations.
In order to do this, we will work on finite jets, on which the existence of $(k,\nu)$-type
foliations are characterized by algebraic subsets $Z(k,\nu)$ and $W(k,\nu)$. By a dimension argument, 
we will show that $\cod(W(k,\nu))>0$.

A first easy Lemma is 

\begin{lemma}\label{TruncationLemma}
The jet $J^{N+2\nu}E$ only depends on $J^{N+1}\Phi$ and $J^{N+\nu}\omega_i$ for $i=1,2$.
\end{lemma}

\begin{proof}If we expand $\Phi$, $\omega_i$ in homogeneous components as follows
$$\Phi=\underbrace{id+\cdots+\Phi^{n-1}}_{\Phi^{<n}}+\Phi^n+\cdots\ \ \ 
\text{and}\ \ \ 
\omega_i=\underbrace{\omega_i^\nu+\ldots+\omega_i^{n_i-1}}_{\omega_i^{<n_i}}+\omega_i^{n_i}+\cdots$$ 
then the first occurence of $\Phi^n$ in the expansion of $E$ is given by
$$
\Phi^*\omega_1 \wedge \omega_2=(\Phi^{<n})^*\omega_1 \wedge \omega_2+\underbrace{(\Lie_{\partial\Phi^n}\omega_1^\nu) \wedge \omega_2^\nu}_{\text{order}\ n+2\nu-1}+\text{ (higher order terms)}
$$
where $\partial\Phi^n$ denotes the vector field $\phi^n\partial_x+\psi^n\partial_y$ and $\Lie_{\partial\Phi^n}$ the corresponding  Lie derivative:
$$\Lie_{\partial\Phi^n}\omega=(\phi^n P_x+\psi^n P_y)dx+(\phi^n Q_x+\psi^n Q_y)dy+Pd\phi^n+Qd\psi^n$$
where $\omega=P dx+Q dy$ (and $P_x=\partial_xP$, \ldots). 
Similarly, the first occurence of the terms $\omega_i^{n_i}$, $i=1,2$, in the expansion of $E$ are given by
$$\begin{matrix}
\Phi^*\omega_1 \wedge \omega_2&= & \Phi^*\omega_1^{<n_1} \wedge \omega_2+\underbrace{\omega_1^{n_1} \wedge \omega_2^\nu}_{\text{order}\ n_1+\nu}+\text{ (higher order terms)}\\
&=& \Phi^*\omega_1 \wedge \omega_2^{<n_2}+\underbrace{\omega_1^\nu \wedge \omega_2^{n_2}}_{\text{order}\ n_2+\nu}+\text{ (higher order terms)}\end{matrix}$$
Therefore, $J^{N+2\nu}E$ depends on $J^n\Phi$ and $J^{n_i}\omega_i$ with $n+2\nu-1=n_i+\nu=N+2\nu$.
\end{proof}

In particular, we have a commutative diagram:
\vskip0.5cm
\xymatrix{
    \Diff\times \Fol \ar[r]^-{E} \ar[dd]_{pr_1} \ar@{.>}[rdd]  & \OO \ar@{.>}[rrdd] && \\
     & & & \\
     \Diff \ar@{.>}[rdd] & J^{N+1}\Diff\times J^{N+\nu}\Fol \ar[rr]^-{J^{N+2\nu}E} \ar[dd]_{pr_1} && J^{N+2\nu}\OO \\ && \\
    & J^{N+1}\Diff & &
}\vskip0.5cm

\noindent where 
$$\Diff=\Phi_0+J_{N_0}^{\infty}\Diffun\ \ \ \text{and}\ \ \ \Fol=\Omega_1^{(k,\nu)}\times\Omega_2^{(k,\nu)}$$ 
and dotted arrows stand for natural projections onto jets. 
After restricting on jets, we get algebraic morphisms between quasi-projective varieties.
Denote by $Z^{N+1}=J^{N+2\nu}E^{-1}(0)$ the preimage of the zero $(N+2\nu)$-jet, and by
$W^{N+1}=\pr_1(Z^{N+1})\subset J^{N+1}\Diff$ its projection of $(N+1)$-jets of diffeomorphisms.

\begin{lemma}\label{lem:strictZariski}
The Zariski closure $\overline{W^{N+1}}\subset J^{N+1}\Diff$ is a strict subset provided that 
$$\rank_p(J^{N+2\nu}E)>\dim(J^{N+\nu}\Fol)$$
at every point $p\in Z^{N+1}$.
\end{lemma}

\begin{proof}
The subset $Z^{N+1}\subset J^{N+1}\Diff\times J^{N+\nu}\Fol$ is algebraically closed. We will denote by $\overline{\dim}(Z^{N+1})$
its maximal dimension.
Its projection $W^{N+1}$ needs not be closed 
as the fiber $\Fol$ of the projection is not compact. However, its Zariski closure $\overline{W^{N+1}}$ satisfies
$$\overline{\dim}(\overline{W^{N+1}})\le\overline{\dim}(Z^{N+1})$$
where $\overline{\dim}(\overline{W^{N+1}})$ is again the maximal dimension:
indeed, $W^{N+1}$ is constructible by Chevalley and therefore locally closed (see \cite[exercises II.3.18-19]{Hartshorne}).
On the other hand, we have 
$$\overline{\dim}(Z^{N+1})+\underline{\rank}(J^{N+2\nu}E)\le\dim(J^{N+1}\Diff)+\dim(J^{N+\nu}\Fol)$$
where $\underline{\rank}(J^{N+2\nu}E)$ is the minimal rank of the map. Indeed, dimension is computed at
smooth points $p$ of $Z^{N+1}$, and at those points, we locally have equality. 
But $\rank_p(J^{N+2\nu}E)$ might decrise outside of smooth points of $Z$. 
Combining the two inequalities, we deduce for the minimal codimension:
$$\begin{matrix}\underline{\cod}(\overline{W^{N+1}})&=&\dim(J^{N+1}\Diff)-\overline{\dim}(\overline{W^{N+1}})\\
&\ge& \underline{\rank}(J^{N+2\nu}E)-\dim(J^{N+\nu}\Fol).
\end{matrix}$$
If the last expression is $>0$, then $\overline{W^{N+1}}$ is a strict subset of $J^{N+1}\Diff$.
\end{proof}

The proof of Proposition \ref{important proposition} will therefore follow from the fact that 
$\underline{\rank}(J^{N+2\nu}E)$ grows faster than $\dim(J^{N+\nu}\Fol)$ while $N\to\infty$.
Precisely, we prove

\begin{lemma}\label{Lem:rank}
For $N\gg 0$, we have 
$$\underline{\rank}_Z(J^{N+2\nu}E)\ \ge\ \underline{\rank}(J^{N+2\nu-1}E)+ N.$$
\end{lemma}

Let us first deduce the 

\begin{proof}[Proof of Proposition \ref{important proposition}]
One easily see that 
$$\dim(J^{N+\nu}\Fol)-\dim(J^{N+\nu-1}\Fol)=4k+4$$
for $N$ large enough. 
Therefore, $\dim(J^{N+\nu}\Fol)\sim 2kN$ when $N\to\infty$.
On the other hand, by Lemma \ref{Lem:rank}, we see that $\underline{\rank}(J^{N+2\nu}E)$
is at least $\sim N^2/2$, so that for $N \gg 0$ large enough we have $\underline{\rank}(J^{N+2\nu}E)>\dim(J^{N+\nu}\Fol)$.
Therefore, by Lemma \ref{lem:strictZariski}, we can set $N_1=N$ and $U=J^{N+1}\Diff\setminus \overline{W^{N+1}}$.
\end{proof}

\begin{proof}[Proof of Lemma \ref{Lem:rank}]
The homogenous part of degree $N+2\nu$ of $E$ is given by
$$
J^{N+2\nu}E - J^{N+2\nu -1}E =(\Lie_{\partial \Phi^{N+1}}\omega_1^\nu)\wedge\omega_2^\nu+\cdots
$$
where dots only depend on $\Phi^{<N+1}$, $\omega_1$ and $\omega_2$. We note that the map 
$$F\ :\ \Phi^{N+1}\mapsto (\Lie_{\partial \Phi^{N+1}}\omega_1^\nu)\wedge\omega_2^\nu$$
is linear. We consider its restriction to the subspace $R$ of those $\Phi^{N+1}=(xH^N,yH^N)$ 
where $H^N$ is a homogeneous polynomial of degree $N$. If we prove that this linear map has rank $\ge N$
in restriction to $R$, then we are done. In restriction to $R$, the linear map becomes
$$F\vert_R\ :\ H\mapsto (\nu+1)H \omega_1^\nu\wedge\omega_2^\nu + (xP_1^\nu+yQ_1^\nu)dH\wedge\omega_2^\nu.$$
Since $\Phi^1=\id$, the compatibility condition of $\omega_1$ and $\omega_2$ in $J^{2\nu+1}E$ give 
$$\omega_1^\nu\wedge\omega_2^\nu=0.$$
Therefore, the linear map writes
$$F\vert_R\ :\ H\mapsto (xP_1^\nu+yQ_1^\nu)dH\wedge\omega_2^\nu.$$
{\bf Assume first} that $xP_1^\nu+yQ_1^\nu\not\equiv0$, i.e. $\ker(\omega_1^\nu)=\ker(\omega_2^\nu)$ do not define the radial foliation.
Then $H$ is in the kernel of the linear map $F\vert_R$ if, and only if
it is a first integral for the foliation $\ker(\omega_2^\nu)$. But in this case, $\ker(F\vert_R)$ has dimension $\le1$, and $F$ has rank $\ge N$. 
Indeed, if $\tilde H$ is another degree $N+1$ homogeneous first integral, then $H/\tilde H$ is also a rational first integral 
for $\ker(\omega_2^\nu)$. But $H/\tilde H$ it is invariant under $(x,y)\mapsto(\lambda x,\lambda y)$, and since  $\ker(\omega_2^\nu)$
is not radial, $H/\tilde H$ must be constant.

{\bf Now assume} that $xP_1^\nu+yQ_1^\nu\equiv0$, i.e. $\omega_i^\nu=H_i\cdot (xdy-ydx)$ with $H_i$ homogeneous 
of degree $\nu-1$ for $i=1,2$. Obviously $H_i\not=0$ and the rank of the linear map $F$ is the same as the rank of 
$$\tilde F\ :\ \Phi^{N+1}\mapsto \Lie_{\partial \Phi^{N+1}}(xdy-ydx)\wedge(xdy-ydx).$$
If we now restrict $F$ to the linear space $L$ of those $\Phi^{N+1}=(\phi,0)$,
then we get $\partial\Phi^{N+1}=\phi\partial_x$ and
$$\tilde F\vert_L\ :\ \phi\mapsto -Ny\phi \cdot dx\wedge dy.$$
Clearly $\ker(\tilde F\vert_L)=0$ and $F$ has rank $N+2$.
\end{proof}

\section{Proof of Theorem \ref{Neighborhood without foliation}}

As a consequence of the previous section we can prove even more. Consider the following norm on power series:
$$\Vert\sum_{i,j\ge0}a_{i,j}x^iy^j\Vert=\mathrm{Sup}\left\{\vert a_{i,j}\vert\right\}\in [0,+\infty).$$
Then we can prove

\begin{thm}
For any $\Phi \in\Diffun$, $N_0 \in \mathbb{N}$ and $\epsilon >0$, there exists $\widehat{\Phi} \in\Diffun$ such that 
\begin{enumerate}
\item $J^{N_0}\Phi =J^{N_0}\widehat{\Phi}$.
\item $\Vert\Phi - \widehat{\Phi}\Vert< \epsilon$.
\item $V_{\widehat{\Phi}}$ does not have any formal foliation.
\end{enumerate} 
\end{thm}

\begin{proof}
Applying proposition \ref{important proposition} to $\Phi_0 = J^{N_0}\Phi$ we have, for each $(k, \nu)$, a natural number $N_1(k, \nu)$ and a Zariski closed set $W(k,\nu) \subset J_{N_0}^{N_1(k,\nu)}\Diffun$ such that, if $J^{N_1(k,\nu)}\widehat{\Phi} = \Phi_0 + \Phi_1$ with $\Phi_1 \in J_{N_0}^{N_1(k,\nu)}\Diffun\setminus W(k,\nu)$ then $V_{\widehat{\Phi}}$ does not have a formal foliation of type $(k, \nu)$. Let us introduce some notation:
\begin{itemize}
\item $\{N_1(k,\nu): k, \nu \geq 0 \} = \{N_1 < N_2 < \ldots \}$, 
\item $I_n = \{ (k, \nu): N_1(k, \nu)= N_n\}$, 
\item $W_n = \bigcup_{(k,\nu) \in I_n} W(k, \nu)$.
\end{itemize}
So far, $W_n$ need not be closed, i.e. might be an infinite union of closed subsets. 
But its complement is clearly of total Lebesgue measure.
Now, we will define the desided diffeomorphism recursively.

\textbf{Step 1: } Define $\widehat{\Phi}_1 = \Phi + \Phi_1$, where $\Phi_1 \in J_{N_0}^{N_1}\Diffun$ is such that $J_{N_0}^{N_1}\Phi + \Phi_1 \notin W_1$ and $\Vert\Phi_1\Vert <\frac{\epsilon}{2}$. The neighborhood $V_{\widehat{\Phi}_1}$ does not
admit formal foliation of any type $(k,\nu)\in I_1$.

\textbf{Step 2: } Let $\pi_2 : J_{N_0}^{N_2}\Diffun \rightarrow J_{N_0}^{N_1} \Diffun$ be the natural restriction. 
Then we set $W_2' = W_2 \cup \pi_2^{-1}(W_1) \subset J_{N_0}^{N_2} \Diffun$. 
Since the complement of this set is dense, we can define $\widehat{\Phi}_2 = \widehat{\Phi}_1 + \Phi_2$ 
where $\Phi_2 \in J_{N_0}^{N_2}\Diffun$ is such that $J_{N_0}^{N_2}\widehat{\Phi}_1 + \Phi_2 \notin W_2'$ 
and $\Vert\Phi_2\Vert <\epsilon / 4$. The neighborhood $V_{\widehat{\Phi}_2}$ does not
admit formal foliation of any type $(k,\nu)\in I_1\cup I_2$.

\textbf{Recursive step: } Let us assume that we have built $\widehat{\Phi}_n$ such that $J^{N_0}\widehat{\Phi}_n = \Phi_0$ and $J^{N_n}_{N_0}\widehat{\Phi}_n \notin W_{n}'$. Writting $\pi_{n+1} : J_{N_0}^{N_{n+1}}\Diffun \rightarrow J_{N_0}^{N_n} \Diffun$ the natural projection and $W_{n+1}' = W_{n+1} \cup \pi_{n+1}^{-1}(W_n') \subset J_{N_0}^{N_{n+1}} \Diffun$ we can define $\widehat{\Phi}_{n+1} = \widehat{\Phi}_n + \Phi_{n+1}$ where $\Phi_{n+1} \in J_{N_0}^{N_{n+1}} \Diffun$ is such that $J_{N_0}^{N_{n+1}}\widehat{\Phi}_n + \Phi_{n+1} \notin W_{n+1}'$ and $\Vert\Phi_{n+1}\Vert <\epsilon / 2^{n+1}$.

It is easy to see that the limit $\widehat{\Phi}=\lim_{n\to\infty}\widehat{\Phi}_n$ has the desired properties.
\end{proof}


\begin{thebibliography}{99}
\frenchspacing

\bibitem{Andreotti}
{\sc A. Andreotti}, 
\emph{Th\'eor\`emes de d\'ependance alg\'ebrique sur les espaces complexes pseudo-concaves.}
Bull. Soc. Math. France {\bf 91} (1963) 1-38. 


\bibitem{FrankMaycol}
{\sc M. Falla Luza, F. Loray},
\emph{Projective structures and neighborhoods of rational curves.}
ArXiv:1707.07868 


\bibitem{FischerGrauert}
{\sc W. Fischer, H. Grauert},
\emph{Lokal-triviale Familien kompakter komplexer Mannigfaltigkeiten.}
Nachr. Akad. Wiss. G\"ottingen Math.-Phys. Kl. II (1965) 89-94.    

\bibitem{Grauert} {\sc H.  Grauert},
\emph{\"Uber Modifikationen und exzeptionelle analytische Mengen.}
Math. Ann. 146 (1962) 331-368. 

\bibitem{Hartshorne} {\sc R. Hartshorne},
\emph{Algebraic geometry}, Graduate Texts in Mathematics, No. 52,
Springer-Verlag, New York-Heidelberg, 1977. xvi+496 pp.
    

\bibitem{HM}
{\sc  H. Hironaka, H. Matsumura},
\emph{Formal functions and formal embeddings.}
J. Math. Soc. Japan 20 (1968) 52-82. 


\bibitem{Hitchin1}
{\sc N. Hitchin},
\emph{Complex manifolds and Einstein's equations.} 
Twistor geometry and nonlinear systems (Primorsko, 1980), 73-99, Lecture Notes in Math., 970, Springer, Berlin-New York, 1982.

\bibitem{Hurtubise}
{\sc J. C. Hurtubise and N. Kamran},
\emph{Projective connections, double fibrations, and formal neighborhoods of lines.}
Math. Ann. 292 (1992) 383-409.

\bibitem{Kodaira} {\sc K. Kodaira}, 
\emph{A theorem of completeness of characteristic systems for analytic families 
of compact submanifolds of complex manifolds.} Ann. of Math. 75 (1962) 146-162.
 
\bibitem{Mishustin}
{\sc M.B. Mishustin},
\emph{Neighborhoods of the Riemann sphere in complex surfaces.}
Funct. Anal. Appl. 27 (1993) 176-185.

\bibitem{Savelev} {\sc V. I. Savel'ev}, 
\emph{Zero-type embeddings of the sphere into complex surfaces.} 
Mosc. Univ. Math. Bull. 37 (1982) 34-39.

\end{thebibliography}
\end{document}